\documentclass[11pt]{article}
\usepackage{amssymb,amsmath,bbm,latexsym,amscd,tikz-cd,enumerate}
\usepackage{amsthm}
\usetikzlibrary{cd}
\usepackage{changes}
\definechangesauthor[name={Felipe}, color=orange]{rev}
\setremarkmarkup{(#2)}
\usepackage{mdframed}
\usepackage{lipsum}
\newmdtheoremenv{theo}{Theorem}

    \newtheorem{theorem}[equation]{Theorem}
    \newtheorem{corollary}[equation]{Corollary}
    \newtheorem{lemma}[equation]{Lemma}
    \newtheorem{proposition}[equation]{Proposition}
\theoremstyle{definition}

        \numberwithin{equation}{section}
\theoremstyle{remark}

\newcommand{\fg}{\mathfrak{g}}

\newcommand{\ot}{\otimes}

\newcommand{\Z}{\mathbb{Z}}
\newcommand{\Cx}{\mathbb{C}}

\newcommand{\Gl}{{\mathcal{G}}}

\newcommand{\la}{\langle}
\newcommand{\ra}{\rangle}
\newcommand{\Mod}[1]{\ (\mathrm{mod}\ #1)}
\newcommand{\Addresses}{{
  \bigskip
  \footnotesize
  Felipe ~Albino dos Santos, \textsc{Departamento de Matem\'atica, Universidade de S\~ao Paulo. S\~ao Paulo - SP, Brasil.}\par\nopagebreak
  \textit{E-mail address}: \texttt{falbinosantos@gmail.com}
}}

\title{On the universal central extension  of superelliptic affine Lie algebras}

\author{Felipe Albino dos Santos\thanks{Funding from CNPq process 142053/2017-1 is gratefully acknowledged.}}

\begin{document}

    \date{}
    \maketitle
    \setcounter{section}{0}

\begin{center}
    \it{Dedicated to the memory of Ben Lewis Cox.}
\end{center}

\begin{abstract}
    Let $p(t)\in \Cx[t]$ be a polynomial with distinct roots. We describe in terms of generators and relations the universal central extension for the infinite dimensional superelliptic affine Lie algebras $\fg\ot R$ with finite dimensional simple Lie algebra $\fg$, whose coordinate ring is of the form $R=\Cx[t,t^{-1},u]$ where $u^m=p(t)$.
\end{abstract}

\section*{Introduction}
    Let $\fg$ be a simple finite-dimensional complex Lie algebra and $\Gl=\fg\ot\Cx[t,t^{-1}]$ the loop algebra of $\fg$ with commutation relations $[x\ot f, y\ot g]=[x,y]\ot fg$, where $x,y\in\fg$ and $f,g\in\Cx[t,t^{-1}]$. We will denote by $\hat{\Gl}$ the universal central extension of $\Gl$, which is the untwisted affine Kac-Moody Lie algebra of $\fg$. In the construction of the loop algebra, we may replace the Laurent polynomial algebra $\Cx[t,t^{-1}]$ by any other commutative associative complex algebra, say $R$, and consider the universal central extension of $\fg\ot R$. When $R$ is the ring of meromorphic functions on Riemann surface with a fixed number of poles, the algebra $\hat{\Gl}$ is called a \emph{current Krichever-Novikov algebra}. These algebras have been studied extensively (see, for example, the book \cite{schlichenmaier2014} and the references therein). The Krichever-Novikov algebras where introduced by Krichever and Novikov in their study of string theory in Minkowski space \cite{krichever1987algebras}, \cite{Krichever1988}. 

    The ring of rational functions on the Riemann sphere regular everywhere except at a finite number of points appears in the study of the tensor module structures for affine Lie algebras in Kazhdan and Luszig's work (see \cite{Kazhdan1991} and \cite{Kazhdan1994}). These algebras are called $N$-point algebras and generalize the untwisted affine Kac-Moody Lie algebras. They are examples of Krichever-Novikov algebras for the genus zero.  Bremner \cite{bremner1994universal} presented the generators and commutation relations of the universal central extension of the $N$-point algebras.  

    Date, Jimbo, Kashiwara and Miwa considered the universal central extension of $\fg\ot \Cx[t,t^{-1},u]$ with $u^2=(t^2-b^2)(t^2-c^2)$ where $b\in\Cx\setminus\{-c,c\}$ in their study of Landau-Lifshitz equation \cite{Date1983}. The algebra above is called the \emph{DJKM algebra}. This is an example of a Krichever-Novikov algebra with genus different from zero. There are interesting and fundamental work has been done by Cox, Futorny and others on the study of the DJKM algebras. The commutation relations in the universal central extension of the DJKM algebras in terms of generators  and families of polynomials were given  in \cite{Cox2011DJKMExtension}. Realizations of the DJKM algebras in terms  of partial differentials operators were constructed in \cite{Cox2014Realizations}. Free field realizations of the DJKM algebras in \cite{Cox2014}.  Study of the universal central extensions of the DJKM algebras led  to the discovery of new families of orthogonal polynomials  in \cite{Cox2013DJKMPolynomials}.

    Another family of examples of Krichever-Novikov algebras is formed by the \emph{elliptic affine Lie algebras}, which are the universal central extension of the Lie algebras $\fg\ot R$ with $R=\Cx[t,t^{-1},u]$ and $u^2=k(t)\in\Cx[t]$ is an elliptic curve. These algebras were studied by Bremner in \cite{bremner1994universal} and \cite{bremner1995four}, where the explicit description in terms of generators, relations and families of polynomials (ultraspherical and Pollaczek) of the commutation relations were given. In the case of Lie algebras of the form $\fg\ot R$ where $R$ is the ring of regular functions defined on an algebraic curve with any number of points removed, Bremner computed the dimension of the associated universal central extension. These results allowed to obtain the free field type realizations of the four point and elliptic affine algebras (see \cite{Cox2014Realizations}, \cite{Cox2011DJKMExtension}, \cite{Cox2016OnAlgebras}). 

    Hyperelliptic affine Lie algebras form a family of Krichever-Novikov Lie algebras with the hyperelliptic algebra $R=\Cx[t,t^{-1},u]$ where $u^2=k(t)\in\Cx[t]$. The hyperelliptic curves are the simplest case of superelliptic curves $u^m=k(t)$, with $k(t)\in\Cx[t]$ and $m\geq 2$. The superelliptic Lie algebra recently have been considered by Cox, Guo, Lu and Zhao in \cite{Cox2017SimpleAlgebras}.  A natural question that arises on the geometric context of algebraic curves is what of the already developed theory and applications of the hyperelliptic curves can be extended to the superelliptic curves (see \cite{Beshaj2015AdvancesApplications} and \cite{Malmendier2019FromCurves}). 
    
    Given that, by Kassel in \cite{kassel1984kahler}, the center $C$ of $\hat{\mathcal{G}}$ is linearly isomorphic to $\Omega_R^1/dR$, the space of K\"ahler differentials of $R$ modulo exact differentials. Bremner stated and answered in \cite{bremner1994universal} three questions about the elliptic Lie algebras:
    \begin{enumerate}
        \item[(1)] describe $C$, in particular determine its dimension;
        \item[(2)] find a basis for $C$, and
        \item[(3)] compute the universal cocycle $\hat{\mathcal{G}}\times\hat{\mathcal{G}}\rightarrow C$ explicitly.
    \end{enumerate}
    
    The purpose of this paper is to generalize some of these results answering these questions for the superelliptic affine Lie algebras $\fg\ot R$ with $R=\Cx[t,t^{-1},u]$, where $u^m=\sum^d_{i=0}a_it^i$ with $a_d=1$, at least one of $a_0, a_1$ different from $0$, and in which 0 has multiplicity $\leq 1$ as root. 
    
    Our main result is the following.
    \\
    {\bf Theorem:}
    \begin{enumerate}
        \item The differentials $\overline{t^{-1}dt}$, together with $\overline{t^{-1}u^ldt}$,$\dots$, $\overline{t^{-d}u^ldt}$ (where we omit $\overline{t^{-d}u^ldt}$ if $a_0=0$), with $l\in\{1,2,\dots,m-1\}$, give finite basis for $\Omega_R^1/dR$
        \item For the resulting superelliptic Lie algebras, we computed the universal cocycle $\hat{\mathcal{G}}\times\hat{\mathcal{G}}\rightarrow C$ explicitly.
    \end{enumerate}
     Applying strategies that could be found in \cite{bremner1994universal}, we found the first part in Theorem \ref{theorem1} and the second part in Theorem \ref{corollary}.

    In the last section, we give well known similar results on hyperelliptic Lie algebras.

\section{Universal Central Extensions.}
    Let $L$ be a Lie algebra and an abelian Lie    algebra $K$, a central extension $\hat{L}$ of $L$ by $K$ is a short exact sequence of Lie algebras
    \begin{equation}
        0\longrightarrow K \xrightarrow{i} \hat{L}\xrightarrow{\pi} L\longrightarrow 0
    \end{equation}
    
    such that $i(K)$ is central in $\hat{L}$. 

    A central extension $\hat{L}$ of $L$ is said to be \emph{universal central extension} if for every central extension $\hat{L}'$ of $L$, there exists a unique pair of Lie algebra homomorphism $(\phi,\phi_0):(K,\hat{L})\mapsto (K',\hat{L}')$ such that

    \begin{center}
        \begin{tikzcd}
            0\arrow{r} \arrow[equal]{d} &K \arrow{r}{i}  \arrow{d}{\phi_0}  & \hat{L} \arrow{d}{\phi} \arrow{r}{\pi} & L\arrow{r}\arrow[equal]{d}  &0 \arrow[equal]{d}\\
            0\arrow{r}  &K' \arrow{r}{i'}  & \hat{L'} \arrow{r}{\pi'} & L\arrow{r}  &0
        \end{tikzcd}
    \end{center}

    commutes.

    We now consider rings $R$ of the form $\Cx[t,t^{-1},u ]$ where $u^m\in \Cx[t]$, with $m\geq2$; thus $R$ has a basis consisting of $t^i$, $t^iu$, $t^iu^2\dots,t^iu^{m-1}$ for $i\in \Z$. We will assume that $u^m=k(t)\in \Cx[t]$ and that $0$ has multiplicity $\leq 1$ as a root of $k(t)$. We write $k(t)=\sum^d_{i=0}a_it^i$ where $a_d=1$ and $a_0,a_1$ are not both 0. The equation $u^m=k(t)$ defines a superelliptic curve. If we write $R^i$ for $\Cx[t,t^{-1}]u^i$, then we see that $R=R^0\oplus R^1\oplus \cdots \oplus R^{m-1}$ is a $\Z/m$-graded ring. 

    Let $\fg$ be a simple finite-dimensional complex Lie algebra. The algebras $\Gl=\fg \ot R$ are examples of \emph{superelliptic loop algebras}. The $\Z/m$-grading induces the structure of a $\Z/m$-graded Lie algebra on $\Gl$ by setting $\Gl^i=\fg\ot R^i$ ($i=0,1,2,\dots,m-1$).

    If we write $\hat\Gl$ for the universal central extension of $\Gl$, then as vector spaces we have $\hat\Gl=\Gl\oplus C$, where $C$ is the kernel of the surjective homomorphism from $\hat\Gl$ onto $\Gl$. That means $C$ is the center of $\hat\Gl$. By Kassel's theorem \cite{kassel1984kahler}, the kernel $C$ is linearly isomorphic to $\Omega_R^1/dR$, the space of K\"ahler differentials of $R$ modulo exact differentials. Our goal is determine a basis for $\Omega_R^1/dR$. 

    Let $F=R\otimes R$ be the left $R$-module with action $f(g\otimes h)=fg\otimes h$ for $f, g, h\in R$. Let $K$ be the submodule generated by the elements $1\otimes fg-f\otimes g-g\otimes f$. Then $\Omega_R^1=F/K$ is the module of K\"ahler differentials. We denote the element $f\otimes g+K$ of $\Omega_R^1$ by $fdg$. We define a map $d:R\rightarrow \Omega_R^1$ by $d(f)=df=1\otimes f+K$ and we denote the coset of $fdg$ modulo $dR$ by $\overline{fdg}$. The commutation relations for $\hat{\Gl}$, the universal central extension of $\Gl$, are
    \begin{align}
        [x\ot f, y\ot g]=[xy]+(x,y)\overline{fdg},&& [x\ot f,\omega]=0,&& [\omega, \omega']=0
    \end{align}
    where $x,y\in\fg$, $f,g\in R$ and $\omega, \omega' \in\Omega_R^1/dR$; here $(x,y)$ denotes the Killing form on $\fg$. All of these objects have a $\Z/m$-grading induced by that on $R$.

    The elements $t^iu^k\ot t^ju^l$, with $i,j\in\mathbb{Z}$ and $k,l\in \{0,1,\dots,m-1\}$ form a basis of $R\ot R$.

    \begin{lemma}\label{lemma1}
        $\Omega_R^1$ is spanned by the differentials $t^iu^kdt$ and $t^iu^ldu$ with $i\in \Z$, $k\in\{0,1,\dots, m-1\}$, and $l\in\{0, 1,\dots, m-2\}$.
    \end{lemma}
    \begin{proof}
        We have to show that any basis element of $R\ot R$ is congruent modulo $K$ to an element in the span of $t^iu^k\ot t$ and $t^iu^l\ot u$ with $i\in \Z$, $k\in\{0,1,\dots, m-1\}$, and $l\in\{0, 1,\dots, m-2\}$.
        We easily show by induction that
        \begin{equation}\label{deriv}
            d(t^ju^l)=jt^{j-1}u^ldt+lt^ju^{l-1}du.
        \end{equation}

        Since $K$ is a submodule of $R\ot R$, we can multiply \eqref{deriv} by $t^iu^k$:
        $$
            t^iu^kd(t^ju^l)= jt^{i+j-1}u^{k+l}dt+lt^{i+j}u^{k+l-1}du.
        $$
        Since $u^{m-1}du=\frac{1}{m}d(u^m) \Rightarrow t^iu^{m-1}du=\frac{1}{m}dt^{i+d-1}dt+\frac{1}{m}(d-1)a_{d-1}t^{i+d-2}dt+\dots+\frac{1}{m}a_1t^idt.$ This shows that any element in the basis of $R \ot R$ is equal to an element in the span of $t^iu^kdt$, and $t^iu^ldu$ with $k\in\{0,1,\dots, m-1\}$, and $l\in\{0,1,\dots, m-2\}$.

    \end{proof}

    \begin{lemma}
        $\Omega_R^1$ is spanned by the differentials $t^idt,t^iudt,\dots, t^iu^{m-1}dt$, with $i\in \Z$, together with $t^{d-1}u^ldu,\dots,tu^ldu,u^ldu$ (where we omit $u^ldu$ if $a_0=0$), with \\ $l$ $\in$ $\{0,1,\dots,m-2\}$.
    \end{lemma}

    \begin{proof}
        We have $\frac{1}{m}ud(u^m)=u^mdu$. Since $u^m=\sum_{k=0}^da_kt^k$ we find that

        \begin{equation}\label{formula0}
            \sum_{k=1}^d\frac{1}{m}ka_kt^{k-1}udt-\sum_{k=0}^da_kt^{k}du=0.
        \end{equation}
        We multiply equation \eqref{formula0} by $t^i$ to get
        \begin{equation} \label{formula1}
            \sum_{k=1}^d\frac{1}{m}ka_kt^{i+k-1}udt-\sum_{k=0}^da_k t^{i+k}du=0.
        \end{equation}
        First assume that $a_0\neq 0$. For $i\geq 0$, formula \eqref{formula1} shows (since $a_d=1$) that $t^{i+d}du$ is equal to a linear combination of $t^{i+d-1}du,\dots,t^idu$ and elements of the form $t^judt$. For $i\leq -1$ it shows (since $a_0\neq 0$) that $t^idu$ is equal to a linear combination of $t^{i+1}du,\cdots,t^{i+d}du$ and elements of the form $t^judt$. From this we show by induction that elements of the form $t^idu$ are equal to a linear combination of $du,\dots, t^{d-1}du$ and elements of the form $t^judt$. 

        If $a_0=0$ then for $i\geq 0$, $t^{i+d}du$ is equal to a linear combination of $t^{i+d-1}du,\dots,$ \\$t^{i+1}du$ and elements of the form $t^judt$, and for $i\leq -1$, since $a_1\neq 0$, $t^{i+1}du$ is equal to a linear combination of $t^{i+2}du,\dots,t^{i+d}du$ and elements of the form $t^judt$. The rest of the argument is similar.

        We can multiply formula \eqref{formula0} by $t^iu^l$  to get
        \begin{equation}\label{formulanova}
            \sum_{k=1}^d\frac{1}{m}ka_kt^{i+k-1}u^{1+l}dt-\sum_{k=0}^da_k t^{i+k}u^ldu=0.
        \end{equation}

        Similarly, from this we show by induction that elements of the form $t^iu^ldu$ are equal to a linear combination of $u^{l}du,\dots, t^{d-1}u^{l}du$ (where we omit $u^{l}du$ if $a_0=0$), and elements of the form $t^ju^{l}dt$. Now Lemma \ref{lemma1} completes the proof. 

    \end{proof}

    \begin{theorem}[\cite{bremner1994universal}, Theorem 2.1]\label{dimensionbremner} 
        The dimension of $\Omega_R^1/dR$ is $2g+n-1$ where $g$ is the genus and $n$ is the number of punctures
    \end{theorem}

    As is well known, the genus of $R$ is $g=\frac{1}{2}(m(d-1)-d-gcd(m,d))+1$ and the number of allowed poles is $n=\gcd(m,d)+1$, if $a_0\neq0$ and $n=\gcd(m,d)+m$ if $a_0=0$ (details can be found in \cite{Hartshorne1977}). Since $\dim\Omega_R^1/dR=2g+n-1$, we will have that 
    \begin{theorem}\label{dimension}
        \begin{equation}
            \dim\Omega_R^1/dR=
            \begin{cases} 
                m(d-1)+1 \textrm{ if } a_0\neq0,\\
                m(d-1)+1-(m-1) \textrm{ if } a_0=0.
            \end{cases}
        \end{equation}
    \end{theorem}
    \begin{theorem}\label{theorem1}
        A finite basis for $\Omega_R^1/dR$  is given by $\overline{t^{-1}dt}$, together with $\overline{t^{-1}u^ldt}$,$\dots$, $\overline{t^{-d}u^ldt}$ (where we omit $\overline{t^{-d}u^ldt}$ if $a_0=0$), with $l\in\{1,2,\dots,m-1\}$.
    \end{theorem}

    \begin{proof}
        The $\Z/m$-grading of $\Omega_R^1$ and $dR$ gives  $\Omega_R^1/dR=\bigoplus_{i=0}^{m-1}(\Omega_R^1)^i/d(R^i)$ where if $k\geq1$ $(\Omega_R^1)^k/d(R^k)=span\la \overline{t^ju^kdt}, \overline{t^{d-1}u^{k-1}du},\dots,\overline{tu^{k-1}du},\overline{u^{k-1}du} \hspace{.2cm}| \hspace{.2cm}j\in\mathbb{Z} \ra$ (we omit $\overline{u^{k-1}du}$ if $a_0=0$) and $(\Omega_R^1)^0/d(R^0):=\la \overline{t^jdt} \hspace{.2cm}| \hspace{.2cm}j\in\mathbb{Z}\ra$.
        
        We first consider the space $(\Omega_R^1)^0/d(R^0)$. We have $d(t^i)=i t^{i-1}dt$ for all $i\in\Z$. From this we see that $t^{i-1}dt\equiv 0\Mod{dR}$ for $i\neq 0$. Therefore $(\Omega_R^1)^0/d(R^0)$ is spanned by $\overline{t^{-1}dt}$.

        Next, we consider the space $(\Omega_R^1)^1/d(R^1)$. This space is spanned by $\overline{t^iudt}$ together with $\overline{t^{d-1}du},\dots,\overline{tdu}$ (and $\overline{du}$ if $a_0\neq0$). We have $d(t^iu)=it^{i-1}udt+t^idu$, and so 
        
        \begin{equation} \label{formula2}
            t^idu\equiv -it^{i-1}udt\Mod{dR}.
        \end{equation}
        
        Thus we only need to consider the elements $t^iudt$. We will show that modulo $dR$ each of these elements is congruent to a linear combination of the finite set listed in the statement of Theorem \ref{theorem1}.

        First suppose that $a_0\neq 0$. We have $t^{i-1}udt\equiv -(1/i)t^{i} du \Mod{dR}$ for $i\neq 0$. By formula \eqref{formula1}, we know that $t^{i}du$ is a linear combination of $t^{i+1}du, \dots,t^{i+d}du$ and $t^iudt,\dots,t^{i+d-1}dt$. Using \eqref{formula2}, we see that $t^idu$, and hence also $t^{i-1}udt$, is congruent modulo $dR$ to an element in the span of $t^iudt,\dots,t^{i+d-1}udt$. Now using induction, we see that for $i\leq -d-1$, the element $t^{i-1}udt$ is congruent modulo $dR$ to a linear combination of $t^{-d}udt,\dots,t^{-1}udt$.

        We also have $t^{i+d-1}udt \equiv -(1/(i+d))t^{i+d}du \Mod{dR}$ for $i\neq -d$. By formula \eqref{formula1} again we know that $t^{i+d}du$ is a linear combination of $t^{i+d-1}du,\dots,t^idu$ and $t^{i+d-1}udt,\dots,t^iudt$. The coefficient of $t^{i+d-1}udt$ in this linear combination is $(d/m)$; hence we can solve for $t^{i+d-1}udt$, showing that it is congruent modulo dR to a linear combination of the same elements (excluding $t^{i+d-1}udt$). By \eqref{formula2} we see that $t^{i+d-1}udt$ is congruent modulo $dR$ to a linear combination of $t^{i+d-2}udt,\dots,t^{i-1}udt$. Now setting $j=i+d-1$ and using induction, we see that for $j\geq 0$ (that is $i\geq -d+1$), the element $t^judt$ is congruent modulo $dR$ to a linear combination of $t^{-1}udt,\dots,t^{-d}udt$.

        The proof in the case $a_0=0$ (and $a_1\neq0$) is similar.

        Then, we consider the spaces $(\Omega_R^1)^l/d(R^l)$ with $l\in\{2,3,\dots,m-1\}$. \\Each $(\Omega_R^1)^l/d(R^l)$ is spanned by $t^iu^ldt$ together with $t^{d-1}u^{l-1}du,\dots,tu^{l-1}du$ (and $u^{l-1}du$ if $a_0\neq 0$). We have that $d(t^{i+1}u^l)=(i+1)t^iu^ldt+lt^{i+1}u^{l-1}du$, then
        \begin{equation}\label{formula2.1}
            t^iu^ldt\equiv \frac{-l}{i+1}t^{i+1}u^{l-1}du \Mod{dR}.
        \end{equation}
        Thus we only need to consider elements $t^iu^ldu$.

        Suppose that $a_0\neq 0$. We have $t^iu^{l-1}du\equiv (-i/l)t^{i-1}u^ldt \Mod{dR}$ for $i\neq 0$. By formula \eqref{formulanova}, we know that $t^{i}u^{l-1}du$ is a linear combination of $t^{i+1}u^{l-1}du, \dots,t^{i+d}u^{l-1}du$ and $t^iu^ldt,\dots,t^{i+d-1}u^ldt$. Using \eqref{formula2.1}, we see that $t^{i+1}u^{l-1}du$, and hence also $t^{i}u^{l}dt$, is congruent modulo $dR$ to an element in the span of $t^iu^ldt,\dots,t^{i+d-1}u^ldt$. Now using induction, we see that for $i\leq -d-1$, the element $t^{i-1}u^ldt$ is congruent modulo $dR$ to a linear combination of $t^{-d}u^ldt,\dots,t^{-1}u^ldt$.

        We also have $t^{i+d-1}u^ldt \equiv -(l/(i+d))t^{i+d}u^{l-1}du \Mod{dR}$ for $i\neq -d$. By formula \eqref{formulanova} again we know that $t^{i+d}u^{l-1}du$ is a linear combination of $t^{i+d-1}u^{l-1}du,\dots,t^iu^{l-1}du$ and $t^{i+d-1}u^ldt,\dots,t^iu^ldt$. The coefficient of $t^{i+d-1}u^ldt$ in this linear combination is $d/m$; hence we can solve for $t^{i+d-1}u^{l-1}dt$, showing that it is congruent modulo dR to a linear combination of the same elements (excluding $t^{i+d-1}u^{l-1}dt$). By \eqref{formula2.1} we see that $t^{i+d-1}u^ldt$ is congruent modulo $dR$ to a linear combination of $t^{i+d-2}u^ldt,\dots,t^{i-1}u^ldt$. Now setting $j=i+d-1$ and using induction, we see that for $j\geq 0$ (that is $i\geq -d+1$), the element $t^ju^ldt$ is congruent modulo $dR$ to a linear combination of $t^{-1}u^ldt,\dots,t^{-d}u^ldt$.

        The proof in the case $a_0=0$ (and $a_1\neq0$) is similar.
        \\~\\
        Then $(\Omega_R^1)^l/d(R^l)$ is spanned by  $\overline{t^{-1}u^ldt}$,$\dots$, $\overline{t^{-d}u^ldt}$ (where we omit $\overline{t^{-d}u^ldt}$ if $a_0=0$). The Theorem \eqref{dimension} completes the proof.
    \end{proof}
    
\section{The commutation relations in $\Omega_R^1/dR$}
    To make the commutation relations for $\hat{\Gl}$ explicit we need to compute $\overline{fdg}$ for any basis elements $f,g\in R$. Note that $\overline{fdg}$ is always the linear combination of basis elements for $\Omega_R^1/dR$ which gives the congruence class of $fdg$ modulo $dR$. By Theorem $\ref{theorem1}$ we know that the elements $\overline{t^{-1}dt}$, together with $\overline{t^{-1}u^ldt}$,$\dots$, $\overline{t^{-d}u^ldt}$ (where we omit $\overline{t^{-d}u^ldt}$ if $a_0=0$), with $l\in\{1,2,\dots,m-1\}$ give a basis for $\Omega_R^1/dR$.
    
    \subsection{Cocycles}
        
        First we give an explicit description of the cocyles contributing to the even part of the superelliptic affine Lie algebra.

        Set 
        \begin{equation}
            \omega_0=\overline{t^{-1}dt} \textrm{ and } \omega_{i,j}=\overline{t^{i}u^jdt} \textrm{ for } j\neq 0.
        \end{equation}

        \begin{proposition}[\cite{bremner1994universal}, Proposition 4.2]
            For $i,j\in\mathbb{Z}$ one has
            $$
                t^id(t^j)=j\delta_{i+j,0}\omega_0.
            $$
        \end{proposition}

        \begin{proposition}
            For $i,j\in\mathbb{Z}$ and $l\in\{1,2,\dots,m-1\}$ we have
            \begin{equation}
                t^iu^ld(t^ju^l)=\left(\frac{j-i}{2}\right) \omega_{i+j-1,2l}.
            \end{equation}
        \end{proposition}
        \begin{proof}
            The congruence follows from the relation
            \begin{align*}
                t^iu^ld(t^ju^l)     =&jt^{i+j-1}u^{2l}dt+lt^{i+j}u^{2l-1}du \\ 
                \equiv& jt^{i+j-1}u^{2l}dt-\left(\frac{i+j}{2}\right)t^{i+j-1}u^{2l}dt \\ 
                =&\left(\frac{j-i}{2}\right)t^{i+j-1}u^{2l}dt.
            \end{align*}
        \end{proof}

        \begin{lemma}[\cite{Cox2011DJKMExtension}, Lemma 2.0.2]
            If $R=\Cx[t,t^{-1},u]$ with $u^m=k(t)\in \Cx[t]$ with degree $d$, then $\Omega_R^1/dR$, one has 
            \begin{equation}\label{equacaocox}
                ((m+1)d+mj)t^{d+j-1}udt\equiv -\sum_{k=0}^{d-1}((m+1)k+mj)a_kt^{k+j-1}udt\Mod{dR}.
            \end{equation}
        \end{lemma}
    
        Thus, 
        \begin{equation}\label{mponto}
            \sum_{k=0}^{d}((m+1)k+mj)a_kt^{k+j-1}udt \equiv 0\Mod{dR}.
        \end{equation}
    
        We define the sequence of polynomials in $d+1$ parameters $Q_{k,l}(a_{d},a_{d-1},\dots,a_0):=\overline{t^{k}u^ldt}$ for $k\in\mathbb{Z}$, $l\in\mathbb{N}\setminus\{0\}$ and $a_{d},a_{d-1},\dots,a_0\in\Cx$ by
        \begin{equation} \label{polinomios}
            \sum_{k=0}^{d}((m+1)k+mj)a_k Q_{k,l}(a_{d},a_{d-1},\dots,a_0) \equiv 0\Mod{dR}.
        \end{equation}

        \begin{proposition}
            Let $Q_k:=Q_k(a_{d},a_{d-1},\dots,a_0)$. For $i,j\in\mathbb{Z}$ we have
            \begin{equation}
                t^{i}u^ld(t^j)\equiv \left(\frac{-j}{d+m(i+j)}\right)\sum_{k=0}^{d-1}a_k Q_{k+d+i+j-3,l} \Mod{dR}.
            \end{equation}
        \end{proposition}

        \begin{proof}
            Given the equation \eqref{mponto}, we have that
            \begin{equation*}
                ((m+1)d+mj)a_kt^{d+j-1}udt \equiv -\sum_{k=0}^{d-1}((m+1)k+mj)a_kt^{k+j-1}udt \Mod{dR}.
            \end{equation*}
            If $j=d+j-1$, then we have
            \begin{equation*}
                t^{j}udt \equiv -\left(\frac{1}{d+m(1+j)}\right)\sum_{k=0}^{d-1}(k+m(k+j-d+1))a_kt^{k+d+j-2}udt \Mod{dR}.
            \end{equation*}
            Since $t^{i}ud(t^j)=jt^{i+j-1}udt$,
            \begin{align*}
                jt^{i+j-1}udt &\equiv -\left(\frac{j}{d+m(1+(i+j-1))}\right)\sum_{k=0}^{d-1}(k+m(k+(i+j-1)-d+1))a_k \\
                & \cdot t^{k+d+(i+j-1)-2}udt \Mod{dR}, and\\
                t^{i}ud(t^j) &\equiv \left(\frac{-j}{d+m(i+j)}\right)\sum_{k=0}^{d-1}(k+m(k+i+j-d))a_kt^{k+d+i+j-3}udt \Mod{dR}.\\
            \end{align*}
        \end{proof}

        We can now give explicit commutation relations for $\mathcal{\hat{G}}$.

        \begin{corollary}\label{corollary}
            The superelliptic affine Lie algebra $\hat{\mathcal{G}}$ has a $\mathbb{Z}/m$-grading in which
        \begin{align*}
            \hat{\mathcal{G}}^0=\fg\ot\Cx[t,t^{-1}]\oplus \Cx\omega_0,&& 
            \hat{\mathcal{G}}^l=\fg\ot\Cx[t,t^{-1}]u^l\bigoplus_{n=1}^d \Cx\omega_{l,n}.
        \end{align*}
        The subalgebra $\hat{\mathcal{G}}^0$ is an untwisted affine Kac-Moody Lie algebra with commutation relations
        \begin{equation*}
            [x\ot t^i,y\ot t^j]=[x,y]\ot \delta_{i+j,0}(x,y)j\omega_0.
        \end{equation*}
        The commutation relations are
        \begin{equation*}
            [x\ot t^iu^l,y\ot t^ju^l]=[x,y]\ot (t^{i+j}u^2)+(x,y)\left(\frac{j-i}{2}\right) \omega_{i+j-1,2l}, 
        \end{equation*}
        if $2l\leq m-1$. When $2l>m-1$,
        \begin{equation*}
            {}[x\ot t^iu^l,y\ot t^ju^l]=[x,y]\ot \left(\sum_{k=0}^d a_kt^{i+j+k}u^{2l-m}\right)+(x,y)\left(\frac{j-i}{2}\right) \omega_{i+j-1,2l}.
        \end{equation*}
        The last commutation relation is
        \begin{equation*}
            [x\ot t^iu^n,y\ot t^ju]=[x,y]\ot (t^{i+j}u^{n+1})+(x,y)\left(\frac{-j}{d+m(i+j)}\right)\sum_{k=0}^{d-1}a_k Q_{k+d+i+j-3,n}.
        \end{equation*}
    \end{corollary}

    \section{The hyperelliptic affine Lie algebras}
        One might want to compare the previous results with well known examples of hyperelliptic affine Lie algebras.
        \subsection{The hyperelliptic case}
        
            We will consider rings $R$ of the form $\Cx[t,t^{-1},u]$ where $u^2\in\Cx[t,t^{-1}]$. Furthermore $p(t)=\sum^d_{i=0}a_it^i\in\Cx[t,t^{-1}]$ where $a_d=1$ and $a_0,a_1$ are not both 0. The equation $u^2=p(t)$ defines a hyperelliptic curve and we call $\fg \otimes R$ an \emph{hyperelliptic loop algebra}.  
        
            \begin{theorem}[\cite{bremner1994universal}, Theorem 3.4]\label{basebremner}
                   A basis of $\Omega_R^1/dR$ is given by $\overline{t^{-1}dt}$ together with $\overline{t^{-1}udt}\dots, \overline{t^{-d}udt}$, where we omit $\overline{t^{-d}udt}$ if $a_0=0$.
            \end{theorem}
            This result was generalized by Theorem \ref{theorem1}. 
            \\
            From now on, we will set
            \begin{align}
                &\omega_0:=\overline{t^{-1}dt}&,  &\omega_-:=\overline{t^{-2}udt}, &\textrm{and}& &\omega_+:=\overline{t^{-1}udt}.&
            \end{align}
        When considering $R=\Cx[t,t^{-1},u]$ where $u^2=(t^2-b)(t^2-c^2)$ we define the Date-Jimbo-Kashiwara-Miwa algebra as $\fg\ot R$. The DJKM algebra is an example of hyperelliptic loop algebra. Using Theorem \ref{basebremner}, it was showed in \cite{Cox2011DJKMExtension} (Theorem 2.0.1) that $\{\overline{t^{-1}dt},\overline{t^{-1}udt},\overline{t^{-2}udt},\overline{t^{-3}udt},\overline{t^{-4}udt}\}$ is a basis for $\Omega_R^1/dR$. It could be verified using Theorem \ref{theorem1}. In \cite{Cox2011DJKMExtension}, Cox and Futorny explicitly described in terms of generators and relations the universal central extension of $\fg\ot R$.
        \subsection{The elliptic case}
            Let $\Sigma$ be a nonsingular compact complex algebraic curve of genus 1. Representing $\Sigma$ as the quotient of complex plane $\Cx$ by the lattice $\Lambda=\mathbb{Z}\oplus\mathbb{Z}\lambda$ with basis $\{1,\lambda\}$ where $Im\lambda>0$. From now on we will restrict the discussion to $R$ that are is the ring of all meromorphic functions on $\Sigma$ which are holomorphic outside the set $\{0,\mu\}$ where $\mu=\frac{1}{2}(1+\lambda)$. The ring $R$ is a ring of elliptic functions.
           
            \begin{proposition}[\cite{bremner1994universal}, Proposition 4.1]
                If $b=-6m/(12m^2-60\sum_{\xi \in \Lambda\setminus\{0\}}\xi^{-4})$, then $R\cong \Cx[t, t^{-1}, u]$ where $u^2=t^3-2bt^2+t$.
            \end{proposition}
            We call the Lie algebra $L(\fg):=\fg\ot R$ the \emph{elliptic loop Lie algebra}. The universal central extension $\hat{\fg}$ of $L(\fg)$, is called the \emph{elliptic affine Lie algebra}.
                
            Bremner realized the universal central extension of $L(\fg)$ and gave a description of the relations satisfied by the basis elements of $\hat{\fg}$. He gave a description of the relations satisfied by the basis elements of $\hat{\fg}$, that could be found using Theorem \ref{basebremner} or \ref{theorem1}. Before show the final result of \cite{bremner1995four}, we recall the 4-parameter \emph{Pollaczek polynomials} $P_k(b)=P_k^{\lambda}(b;\alpha,\beta,\gamma)$ that are defined as a family of polynomials satisfying the recursion formula
            \begin{equation}
                (k+\gamma)P_k(b)=2[(k+\lambda+\alpha+\gamma-1)b+\beta]P_{k-1}(b)-(k+2\lambda+\gamma-2)P_{k-2}(b).
            \end{equation}
            Define two sequences of polynomials $p_k(b)$, $q_k(b)$ for $k\in\mathbb{Z}$ by
            \begin{equation}
                \overline{t^{k-2}udt}=p_k(b)\overline{t^{k-1}udt}+q_k(b)\overline{t^{-2}udt}.
            \end{equation}
            \begin{lemma}[\cite{bremner1994universal}, Lemma 4.4]
                The polynomials $p_k(b)$ and $q_k(b)$ are Pollaczek polynomials for the parameter values $\lambda=-1/2$, $\alpha=0$, $\beta=-1$ and $\gamma=1/2$. The initial conditions are
                \begin{align}
                    p_0(b)=0,&& p_1(b)=1,&& q_0(b)=1,&&\textrm{ and }&& q_1(b)=0.
                \end{align}
            \end{lemma}
            The final result of \cite{bremner1994universal} that gave the commutation relations for the elliptic affine Lie algebra is
            \begin{theorem}[\cite{bremner1994universal}, Theorem 4.6]
                The elliptic affine Lie algebra $\hat{\fg}$ has a $\mathbb{Z}/2\mathbb{Z}$-grading in which
            \begin{align*}
                \hat{\fg}^0=\fg\ot\Cx[t,t^{-1}]\oplus \Cx\omega_0,&& &\hat{\fg}^1=\fg\ot\Cx[t,t^{-1}]u\oplus \Cx\omega_{-}\oplus \Cx\omega_{+}.
            \end{align*}
            For $x,y\in\fg$ the commutation relations defining $\hat{\fg}$ are
            \begin{align*}
                [x\ot t^{i-1}u,y\ot t^{j-1}]&=[x,y]\ot (t^{i+j-1}-2bt^{i+j}+t^{i+j+1})+
                (x,y)\begin{cases}
                -2jb\omega_{0}, \textrm{ for } i+j=0 \\
                \frac{1}{2}(j-i)\omega_{0}, \textrm{ for } |i+j|=1\\
                0, \textrm{ for } |i+j|\geq 2.
                \end{cases}
                \\
                {}[x\ot t^{i-1}u,y\ot t^{j}u]&=[x,y]\ot t^{i+j-1}u+(x,y)j(p_{|i+j|}(b)\omega_++q_{|i+j|}(b)\omega_-).
            \end{align*}
        \end{theorem}
\subsection{The 4-point case}

        Let $R_a$ be the $4$-point ring $\Cx[s,s^{-1},(s-1)^{-1},(s-a)^{-1}]$, $a\in\Cx\setminus\{0,1\}$, and $S_b=\Cx[t,t^{-1},u]$, where $u^2=t^2-2bt+1$ with $b$ a complex number not equal to $\pm1$. 
        \begin{proposition}[\cite{bremner1995four}, Proposition 1.1]
            If $b=(a+1)/(a-1)$ with $a\in \Cx\setminus\{0,1\}$, then $R_a\cong S_b$ and $b\neq \pm 1$.
        \end{proposition}
        Here $\fg$ still a simple finite-dimensional complex Lie algebra. We call the Lie algebra $L(\fg):=\fg\ot R_a$ the \emph{4-point loop Lie algebra}. The universal central extension of $L(\fg)$, is called the \emph{4-point affine Lie algebra} and it will be denoted $\hat{\fg}$.
            
        Bremner realized the universal central extension of $L(\fg)$ and gave a description of the realizations satisfied by the basis elements of $\hat{\fg}$.
           \begin{theorem}[\cite{bremner1995four}, Theorem 3.6]
               The space $\Omega_{R_a}^1/dR_a$ has basis $\{\omega_0, \omega_-, \omega_+\}$.
        \end{theorem}

        Bremner gave a description of the relations satisfied by the basis elements of $\hat{\fg}$, that are $x\ot t^n$. $x\ot t^nu$, $\omega_0$, and $\omega_\pm$. We recall the \emph{ultraspherical (Gegenbauer) polynomials} $P_k^{\lambda}(b)$, which are defined to be the coefficient to $t^k$ in the Taylor series of $P^{-\lambda}(b,z)=(1-2bt+t^2)^{-\lambda}$. Setting $\lambda=-\frac{1}{2}$, $P_k(b)=P_k^{-1/2}(b)$ and $P=P_k^{-1/2}(b,z)$. Define for $b\neq\pm1$
        \begin{equation}
            Q_k(b):=-\frac{P_k+2(b)}{b^2-1}
        \end{equation}
        which is a polynomial in $b$. The final result of \cite{bremner1995four} is
        
        \begin{corollary}[\cite{bremner1995four}, Theorem 3.6]
            The 4-point affine Lie algebra $\hat{\fg}$ has a $\mathbb{Z}/2\mathbb{Z}$-grading in which
        \begin{align*}
            \hat{\fg}^0=\fg\ot\Cx[t,t^{-1}]\oplus \Cx\omega_0,& &\hat{\fg}^1=\fg\ot\Cx[t,t^{-1}]u\oplus \Cx\omega_{-}\oplus \Cx\omega_{+}.
        \end{align*}
        For $x,y\in\fg$ the commutation relations defining $\hat{\fg}$ are
        \begin{align*}
            [x\ot t^{i-\frac{1}{2}},y\ot t^j]&=[x,y]\ot t^{i+j-\frac{1}{2}}+\delta_{i+j} \omega_{0}, 
            \\
            {}[x\ot t^{i-\frac{1}{2}}u,y\ot t^{j-\frac{1}{2}}u]&=[x,y]\ot (t^{i+j-1}-2bt^{i+j}+t^{i+j+1})\\
            &+(x,y)\omega_0\left(-2jb\delta_{i+j,0}+\frac{1}{2}(j-i)(\delta_{i+j,-1}+\delta_{i+j,1})\right)\\
            {}[x\ot t^{i-\frac{1}{2}}u,y\ot t^j]&=[x,y]\ot (t^{i+j-\frac{1}{2}}u)+(x,y)j\left(Q_{i+j-\frac{3}{2}}(b)(b\omega_++\omega_-)\delta_{i+j\geq\frac{3}{2}})\right.\\
            &\left.+\omega_\pm\delta_{i+j,\pm\frac{1}{2}}+Q_{-i-j-\frac{3}{2}}(b)(\omega_++b\omega_0)\delta_{i+j\leq-\frac{3}{2}}\right).
        \end{align*}
    \end{corollary}
    
    When considering $R=\Cx[t,t^{-1},u]$ where $u^2=t^2+4t$ we define the 3-point algebra as $\fg\ot R$. The 3-point loop algebra is an example of elliptic loop algebra. Using Theorem \ref{basebremner}, it was showed in \cite{Cox2014Realizations} (Proposition 2.2) that $\{\overline{t^{-1}dt},\overline{t^{-1}udt}\}$ is a basis for $\Omega_R^1/dR$. It could be verified using Theorem \ref{theorem1}. In \cite{Cox2014Realizations}, Cox and Jurisich described the universal central extension of the 3-point current algebra $\mathfrak{sl}(2,R)$ and constructed realizations of it in terms of sums of partial differential operators.

\bibliographystyle{alpha}

\begin{thebibliography}{DJKM83}

\bibitem[Bre94]{bremner1994universal}
Murray Bremner.
\newblock {Universal central extensions of elliptic affine Lie algebras}.
\newblock {\em Journal of Mathematical Physics}, 35(12):6685--6692, 1994.

\bibitem[Bre95]{bremner1995four}
Murray Bremner.
\newblock {Four-point affine Lie algebras}.
\newblock {\em Proceedings of the American Mathematical Society},
  123(7):1981--1989, 1995.

\bibitem[BSZ15]{Beshaj2015AdvancesApplications}
Lubjana Beshaj, Tony Shaska, and Eustrat Zhupa.
\newblock {Advances on Superelliptic Curves and their Applications}.
\newblock {\em Advances on Superelliptic Curves and their Applications}, 41,
  2015.

\bibitem[CF11]{Cox2011DJKMExtension}
Ben Lewis Cox and Vyacheslav Futorny.
\newblock {DJKM algebras I: Their universal central extension}.
\newblock {\em Proceedings of the American Mathematical Society},
  139(10):3451--3451, 2011.

\bibitem[CFM14]{Cox2014}
Ben Lewis Cox, Vyacheslav Futorny, and Renato~Alessandro Martins.
\newblock {Free field realizations of the date-jimbo-kashiwara-miwa algebra}.
\newblock {\em Developments in Mathematics}, 38:111--136, 2014.

\bibitem[CFT13]{Cox2013DJKMPolynomials}
Ben Lewis Cox, Vyacheslav Futorny, and Juan~A. Tirao.
\newblock {DJKM algebras and non-classical orthogonal polynomials}.
\newblock {\em Journal of Differential Equations}, 255(9):2846--2870, 2013.

\bibitem[CGLZ17]{Cox2017SimpleAlgebras}
Ben Lewis Cox, Xiangqian Guo, Rencai L{\"{u}}, and Kaiming Zhao.
\newblock {Simple superelliptic Lie algebras}.
\newblock {\em Communications in Contemporary Mathematics}, 19(3), 2017.

\bibitem[CJ14]{Cox2014Realizations}
Ben Lewis Cox and Elizabeth Jurisich.
\newblock {Realizations of the three-point Lie algebra sl(2,R)⊕({$\Omega$}R
  /dR )}.
\newblock {\em Pacific Journal of Mathematics}, 270(1):27--48, 2014.

\bibitem[Cox16]{Cox2016OnAlgebras}
Ben Lewis Cox.
\newblock {On the universal central extension of hyperelliptic current
  algebras}.
\newblock {\em Proceedings of the American Mathematical Society},
  144:2825--2835, 2016.

\bibitem[DJKM83]{Date1983}
Etsuro Date, Michio Jimbo, Masaki Kashiwara, and Tetsuji Miwa.
\newblock {Landau-Lifshitz equation: Solitons, quasi-periodic solutions and
  infinite-dimensional Lie algebras}.
\newblock {\em Journal of Physics A: Mathematical and General}, 16(2):221--236,
  1983.

\bibitem[Har77]{Hartshorne1977}
Robin Hartshorne.
\newblock {\em {Algebraic Geometry}}, volume~52 of {\em Graduate Texts in
  Mathematics}.
\newblock Springer New York, New York, NY, 1977.

\bibitem[Kas84]{kassel1984kahler}
Christian Kassel.
\newblock {K{\"{a}}hler differentials and coverings of complex simple Lie
  algebras extended over a commutative algebra}.
\newblock {\em Journal of Pure and Applied Algebra}, 34(2):265--275, 1984.

\bibitem[KL91]{Kazhdan1991}
David Kazhdan and George Lusztig.
\newblock {Affine lie algebras and quantum groups}.
\newblock {\em International Mathematics Research Notices}, 1991.

\bibitem[KL94]{Kazhdan1994}
David Kazhdan and George Lusztig.
\newblock {Tensor structures arising from affine Lie algebras. IV}.
\newblock {\em Journal of the American Mathematical Society}, 1994.

\bibitem[KN87]{krichever1987algebras}
Igor~Moiseevich Krichever and Sergei~Petrovich Novikov.
\newblock {Algebras of virasoro type, riemann surfaces and structures of the
  theory of solitons}.
\newblock {\em Functional Analysis and Its Applications}, 21(2):126--142, 1987.

\bibitem[KN88]{Krichever1988}
Igor~Moiseevich Krichever and Sergei~Petrovich Novikov.
\newblock {Virasoro type algebras, Riemann surfaces and strings in Minkowski
  space}.
\newblock {\em Funkts. Anal. Prilozhen.}, 21(4):47--61, 1988.

\bibitem[MS19]{Malmendier2019FromCurves}
Andreas Malmendier and Tony Shaska.
\newblock {From hyperelliptic to superelliptic curves}.
\newblock {\em Albanian Journal of Mathematics}, 13(1):107–200, 2019.

\bibitem[Sch14]{schlichenmaier2014}
Martin Schlichenmaier.
\newblock {\em {Krichever-Novikov Type Algebras}}.
\newblock Walter de Gruyter GmbH, Luxembourg, 2014.

\end{thebibliography}

\Addresses

\end{document}